\documentclass[english]{smfart}
\usepackage{a4wide}

\usepackage[T1]{fontenc}   
\usepackage[latin1]{inputenc}
\usepackage[francais]{babel} 
\usepackage{amsmath}
\usepackage{amsfonts}
\usepackage{dsfont}
\usepackage{amssymb}
\usepackage{smfthm}
\usepackage{mathrsfs}
\usepackage{color}
\usepackage{enumerate}
\usepackage{hyperref}
\usepackage[pdftex,all]{xy}
\usepackage{tikz}

\theoremstyle{plain}
\newtheorem{thm}{Theorem}
{\bfseries}{\itshape}
{\bfseries}{\itshape}
\newtheorem{proposition}[thm]{Proposition}

\newtheorem{lem}[thm]{Lemma}

\newtheorem{question}{Question}
\theoremstyle{definition}
\newtheorem{definition}[thm]{Definition}
\newtheorem{example}[thm]{Example}

\theoremstyle{remark}
\newtheorem{remark}[thm]{Remark} \renewcommand{\leq}{\leqslant} 
\renewcommand{\geq}{\geqslant}

\newcommand{\F}{\ensuremath{\mathbf{F}}}
\newcommand{\Fq}{\ensuremath{\mathbf{F}_q}}

\newcommand{\code}[1]{\ensuremath{\mathcal{#1}}}

\newcommand{\CC}{\code{C}}

\newcommand{\CL}[3]{\mathcal C_L({#1},{#2},{#3})}
\newcommand{\cP}{\mathcal{P}}

\renewcommand{\curve}[1]{\ensuremath{\mathscr{#1}}}
\def\X{\curve{X}}
\def\Y{\curve{Y}}
\newcommand{\HC}{\mathscr{H}_2(\CC)}

\renewcommand{\P}{\mathbf{P}}

\newcommand{\word}[1]{\ensuremath{\boldsymbol{#1}}}
\newcommand{\av}{\word{a}}
\newcommand{\bv}{\word{b}}
\newcommand{\cv}{\word{c}}

\newcommand{\ev}{\word{e}}

\newcommand{\xv}{{\word{x}}}
\newcommand{\yv}{{\word{y}}}

\newcommand{\map}[4]{\left\{
\begin{array}{ccc}
#1 & \longrightarrow & #2 \\ #3 & \longmapsto & #4 
\end{array}
\right.}

\def\ie{{\em i.e}}
\newcommand{\eqdef}{\stackrel{\text{def}}{=}} \title[How arithmetic and geometry make codes better]{How arithmetic
  and geometry\\ make error correcting codes better}

\author{Alain Couvreur}
\address{Inria}
\address{Laboratoire LIX, CNRS UMR 7161, \'Ecole Polytechnique,
  91120 Palaiseau, France}
\email{alain.couvreur@inria.fr}

 \begin{abstract}
  This note completes a talk given at the conference {\em
    Curves over Finite Fields: past, present and future} celebrating
  the publication the book {\em Rational Points on Curves over Finite
    Fields} by J.-P. Serre and organised at Centro de ciencias de
  Benasque in june 2021. It discusses a part of the history of
  algebraic geometry codes together with some of their recent
  applications. A particular focus is done on the ``multiplicative''
  structure of these codes, {\em i.e.} their behaviour with respect
  to the component wise product. Some open questions are
  raised and discussed.
\end{abstract}
 
\begin{document}
\maketitle

\tableofcontents

\section*{Introduction}
The story of algebraic geometry codes is a fascinating one for both
mathematicians and computer scientists. It lies at the crossroad of
coding theory, number theory and algebraic geometry and testifies of
the need for such interactions between different communities to
provide efficient but also elegant solutions to practical problems.

Coding theory began in the 50's after the works of Claude Shannon
\cite{shannon1948} and Richard Hamming \cite{hamming1950}.  The former
developing a stochastic approach for error correction and building the
foundations of information theory and the latter proposing a more
combinatorial point of view. The original motivations came from the need
for error correction in computations and communications.  Our history
is full of technologies who would never have been as efficient as they
were without error correcting codes.  One can for instance mention
NASA's program Mariner 9 whose objective was to transmit pictures from
Mars and where the error tolerance has been guaranteed by Reed--Muller
codes.

Today, coding theory has become a rich and wide mathematical theory
with strong connections with combinatorics, probabilities, arithmetic,
geometry and algorithmics. More importantly, the field of applications
of coding theory impressively broadened in 70 years and introducing
codes as objects permitting good communications across noisy channels
would be indisputably a very restrictive description. Coding theory
finds now numerous other applications, for instance in symmetric
cryptography, in public key cryptography or in cloud storage. Results
from coding theory are also used to address problems in other
mathematical areas such as combinatorics or number theory, or in
theoretical computer science with famous results such as the
celebrated PCP theorem \cite{ALMSS98}.

Inside coding theory, algebra always played a particular rule.  The
use of algebraic structures to design codes is almost as old as coding
theory itself. Reed--Muller codes \cite{reed54} have been designed
using multivariate polynomials over $\F_2$, then Reed--Solomon ones
\cite{reed60siam} are constructed thanks to univariate polynomials and
Berlekamp \cite{B68,B15} developed a decoding algorithm based on
algebraic features. Interestingly, Berlekamp's interest on algebraic
codes encouraged him to focus on some computational aspects of finite
fields which lead him to discover his famous algorithm for factorizing
polynomials over finite fields. In summary, algebra had an
indisputable impact on the development of coding theory but the
converse is also true.

Algebraic geometry enters the game in the early 80's with Goppa's
seminal article \cite{goppa1981dansssr}, followed by the unexpected
and rather counter intuitive result proved by Tsfasman, Vl\u{a}du\c{t}
and Zink \cite{tsfasman1982mn} showing that, some codes from modular
or Shimura curves have better parameters than random codes. This
breakthrough motivated the impressive development of the theory of
algebraic geometry codes which today represent a broad and rich
research area with a significant number of independent applications.
The present article presents a part of this story. Then a particular
focus is done on the behaviour of algebraic geometry codes with respect
to the component wise product and its
various applications, to decoding, cryptanalysis or secret sharing.

 \section{Prerequisites in coding theory}\label{sec:coding}
The central objects of our study are linear error correcting codes.
In the sequel, we denote by {\em code} any linear subspace $\CC$ of
$\Fq^n$ endowed with the {\em Hamming metric}. The latter being
defined as
\[
  \forall \xv, \yv \in \Fq^n,\qquad d_H(\xv, \yv) \eqdef
  \Big| \Big\{i\in \{1,\dots, n\} ~|~ x_i \neq y_i\Big\} \Big|.
\]
The {\em Hamming weight} of a vector $\xv \in \Fq^n$ is defined
as $w_H (\xv) \eqdef d_H(\xv, 0)$. 

Given a code $\CC$, its {\em parameters} denoted $[n,k,d]_q$ are
respectively:
\begin{itemize}
\item its {\em length} $n$, {\ie} the dimension of its ambient space $\Fq^n$;
\item its {\em dimension} $k$ as an $\Fq$--linear space;
\item its {\em minimum distance} $d_{\text{min}}(\CC)$ or $d$ when
  there is no ambiguity on the considered code, is the minimum Hamming
  distance between two distinct elements of $\CC$ which can be
  formally defined as
  \[
    d_{\text{min}}(\CC) \eqdef \min_{\cv \in \CC \setminus \{0\}}
    \Big|\{i \in \{1, \dots, n\} ~|~ c_i \neq 0\}\Big|.
  \]
\end{itemize}
A {\em decoding algorithm} takes as input an ambient vector
$\yv \in \Fq^n$ and outputs a word $\cv$ close to $\yv$ (preferently
the closest one) or fails (for instance if $\yv$ is too far from any
element of $\CC$).  Finally, given a code $\CC \subseteq \Fq^n$ one
sometimes need to consider its {\em dual} $\CC^\perp$ which is its
orthogonal space with respect to the canonical inner
product: \[\langle \xv , \yv \rangle \eqdef \sum_{i=1}^n x_iy_i.\]

If it is elementary to generate a code, the estimate of its minimum
distance such as the design of an efficient (let us say polynomial
time) decoding algorithm are very difficult tasks\footnote{The
  decoding problem has even been proved to be NP--complete in
  \cite{berlekamp1978it}.} which turn out to be intractable for the
majority of the codes.  Actually, the aforementioned problems are
solvable only for very particular codes which are equipped with a
particular structure. For this sake, algebra provides remarkable
constructions. The most elementary one being that of Reed and Solomon
\cite{reed60siam}.

\begin{definition}[Reed--Solomon codes]
  Let $\xv = (x_1, \dots, x_n)$ be an $n$--tuple of distinct elements
  of $\Fq$ and a positive integer $k<n$, we define
  \[
    RS_k(\xv) \eqdef \{(f(x_1), \dots, f(x_n))) ~|~ f \in \Fq[X],\ \deg f < k\}.
  \]
\end{definition}

Such a code is known to have parameters $[n,k,n-k+1]$ and the proof of
that fact rests essentially on the elementary property that a nonzero
polynomial of degree $<k$ has less than $k$ distinct roots. Moreover,
these codes are optimal in the sense that they reach the so called
{\em Singleton bound} asserting that any $[n,k,d]$ code satisfies
\begin{equation}\label{eq:Singleton}
  d \leq n-k+1.
\end{equation}
Finally, after the works of Berlekamp \cite{WB83,B15}, such codes
benefit from a decoding algorithm correcting in polynomial time an
amount of errors up to half the minimum distance. Using fast
arithmetic, decoding can actually be performed in quasi--linear time.
Finally, after the works of Sudan and Guruswami--Sudan in the late
90's \cite{sudan97jcomp,guruswami99it} it is now known that correcting
efficiently an amount of errors beyond half the minimum distance is
possible for such codes at the cost of possibly (and rarely) returning
a list of closest vectors instead of a single one.

 \section{Algebraic geometry codes}
Reed--Solomon codes are also known to have a drawback: they should be
defined over a field which is at least as large as the length of the
code. On the other hand, many computer science applications look for
long codes over very small fields. This issue has been addressed in
the early 80's by the Russian mathematician V.D.~Goppa
\cite{goppa1981dansssr} who suggested to replace polynomials by
functions\footnote{Actually, Goppa's original publication proposed to
  design codes by evaluating residues of some differential forms
  living in some Riemann--Roch space of differentials. However, its
  presentation can easily be reformulated into an equivalent one with
  rational functions.} in a Riemann--Roch space of a given curve $\X$
and the evaluation values $x_1, \dots, x_n$ by some rational points of
$\X$.

\begin{definition}\label{def:AGcodes}
  Let $\X$ be a smooth projective absolutely irreducible curve
  over a field $\Fq$. Let $\cP = (P_1, \dots, P_n)$ be an ordered
  $n$--tuple of distinct rational points of $\X$ and $G$ be a rational
  divisor on $\X$ whose support avoids\footnote{This restriction on the
    support of $G$ can easily be relaxed. See
    \cite[Rem. 15.3.8]{couvreur2021}.} $\cP$. The algebraic geometry
  code (AG code) associated to the triple $(\X, \cP, G)$ is defined as
  \[
    \CL{\X}{\cP}{G} \eqdef \{(f(P_1), \dots, f(P_n)) ~|~ f\in L(G)\},
  \]
  where $L(G)$ denotes the Riemann--Roch space associated to $G$.
\end{definition}

This code construction is of deep interest for many reasons.
First, because of their excellent parameters and second, because
the algebraic decoding algorithms existing for Reed--Solomon
naturally extend to AG codes.

\begin{proposition}\label{prop:param_AG_codes}
  In the context of Definition~\ref{def:AGcodes}, let $g$ be the genus of $\X$.
  Then, the code
  $\CL{\X}{\cP}{G}$ has parameters $[n,k,d]$
  where
    \begin{align*}
      k & \geq \deg G + 1 - g \quad \text{with equality if} \quad
          \deg G > 2g-2; \\
      d & \geq n - \deg G.
    \end{align*}
\end{proposition}

Reed--Solomon codes are nothing but codes from $\P^1$ and,
as already mentioned, their main drawback  is that
for a fixed base field $\F_q$ the length of a Reed--Solomon code is
bounded from above by $|\P^1(\Fq)| = q+1$. On the other hand, for AG codes, the
length is allowed to reach the number $|\X (\Fq)|$ of rational points,
which can be arbitrarily large. Considering sequences of curves
whose number of rational points goes to infinity, it is possible to
design sequences of codes whose length goes to infinity while the
dimension and the minimum distance are linear in the length. Indeed,
regarding Proposition~\ref{prop:param_AG_codes}, the parameters of an
AG code satisfy
\[
  k+d \geq n + 1 - g.
\]
Set $R \eqdef \frac k n$ and $\delta \eqdef \frac d n$
then, 
\begin{equation}
  R+\delta \geq 1 + \frac 1 n- \frac{g}{n}\cdot 
\end{equation}
Introducing the {\em Ihara constant}
\[
  A(q) \eqdef \limsup_{g \rightarrow + \infty} \max_{\X} \frac{|\X (\Fq)|}{g},
\]
where the max is taken over all the curves of genus $g$, we deduce the
existence of sequences of codes whose asymptotic rates and relative distances
satisfy
\begin{equation}\label{eq:TVZbound}
  R+\delta \geq 1 - \frac{1}{A(q)}\cdot
\end{equation}
In the early 80's, Tsfasman, Vl\u{a}du\c{t} and Zink
\cite{tsfasman1982mn} and independently Ihara
\cite{ihara1981jfsuTokyo} proved that $A(q) \geq \sqrt q - 1$ when $q$
is a square and this quantity is reached by some sequences of modular
curves for $q = p^2$ with $p$ prime and by some sequences of Shimura
curves for general squares $q$.  Later on, Drinfeld and Vl\u{a}du\c{t}
\cite{vladut1983faa} proved that $A(q) \leq \sqrt q - 1$ for any prime
power $q$, showing that the aforementioned lower bound for $q$ square
was actually optimal.  These results lead to an impressive
breaktrhough in coding theory since for a square $q \geq 49$ this
lower bound exceeds the famous Gilbert Varshamov bound which is
safisfied by ``almost any code''.  In short, some sequences of AG
codes are asymptotically better than random codes, a particularly
striking and unexpected result.  Inequality \eqref{eq:TVZbound} is
usually referred to as the {\em Tsfasman--Vl\u{a}du\c{t}--Zink bound}
and, because of its famous coding theoretic consequences, has been the
first stone of an impressive line of works in order to obtain better
estimates for the Ihara constant when $q$ was not a square. Hundreds
of papers have been published in this direction by using either class
field towers or the so--called {\em recursive towers} introduced by
Garcia and Stichtenoth \cite{GS1995}.
Actually, the point of this paper is not to develop the
presentation of these lines of works which would require a fully
independent study. See for instance \cite{beelen22}.
The objective of the previous discussion was only to
explain why AG codes became so famous in the 80's.  They now represent
a wide research area and it is worth mentioning that their
interest and their applications actually exceed the single question of
which optimal pairs $(R, \delta)$ are reachable.

Indeed, in many situations where Reed--Solomon codes have been proposed
as a first solution, AG codes appear immediately after as a more
flexible option permitting to circumvent the issues raised by the
inherent drawback of Reed--Solomon codes~: their bounded length for a
fixed alphabet. This appears for instance in cloud computing where
{\em Locally Recoverable Codes} (LRC) are introduced in order to
address the practical issues raised by server failures in data
centers.  In \cite{tamo2014it}, Tamo and Barg propose a construction
derived from Reed--Solomon codes which turns out to be optimal with
respect to some Singleton--like bound but requiring the length to be
bounded. Later on, Barg, Tamo and Vl\u{a}du\c{t} \cite{barg2017it}
propose an AG--based construction providing the expected flexibility
in terms of blocklength. We refer the reader to
\cite[\S~15.9]{couvreur2021} for further details.

In the sequel, we discuss other coding theoretic issues where AG codes
have been involved. We also discuss open coding theoretic questions
where the use of algebraic geometry could be crucial.

 \section{Component wise product}\label{sec:starprod}
In the last decades, an apparently elementary operation which took an
increasing importance in coding theory is the component wise product
of vectors or codes, {\em i.e.} the multiplication operation in
$\Fq^n$ regarded as a product of rings:
\[
  \forall \av, \bv \in \Fq^n, \quad \av \star \bv \eqdef (a_1b_1,
  \dots, a_nb_n).
\]
This operation extends to codes where, to keep a linear structure,
the product of two codes is defined as the span of the products of
their vectors
\begin{definition}
  Given two codes $\code{A}, \code{B} \subseteq \Fq^n$,
  \[
    \code A \star \code B \eqdef \text{Span}_{\Fq} \left\{ \av \star \bv ~|~
    \av \in \code A, \bv \in \code B\right\}.
  \]
\end{definition}
Importing such a ring structure over $\Fq^n$ is actually rather
natural when studying codes such as AG codes which are obtained by
evaluation of functions living in a ring. Reconsidering
Definition~\ref{def:AGcodes} and denoting by $\mathcal O_{\X,\cP}$ the
semi-local ring $\bigcap_{P\in \cP} \mathcal O_{\X, P}$, the evaluation map
\[
  \text{ev}_{\cP} : \map{\mathcal O_{\X, \cP}}{\Fq^n}{f}{(f(P_1),
    \dots, f(P_n))}
\]
is a ring homomorphism from $\mathcal O_{\X, \cP}$ to $\Fq^n$ equipped
with the $\star$ product. Moreover, $\CL{\X}{\cP}{G}$ is nothing but
the image of $L(G)$ by this map. In summary, considering the $\star$
product is natural in order to take into account the rich algebraic
structure of AG codes coming from the parent curve.

In the last decades the $\star$ product took an increasing importance
in coding theory. For instance, we explain in \S~\ref{sec:decoding}
and \ref{sec:mpc} that the design of algebraic decoding algorithms or
the properties of some secret sharing schemes can easily be
interpreted in terms of this operation. The reader interested in an
in--depth study of the arithmetic of this operation together with some
of its applications is encouraged to look at the excellent survey
\cite{HR-AGCT14}.

For now, and because of the underlying algebraic geometric structure,
AG codes have a very peculiar behaviour with respect to the $\star$
product. Indeed, for a general code $\CC \subseteq \Fq^n$, we have
\begin{equation}\label{eq:random_square}
  \dim \CC \star \CC \leq \min \left\{n,\ {\dim \CC + 1 \choose
      2}\right\}.
\end{equation}
Moreover, it is proved in \cite{CCMZ2015} that for a random code the
above inequality is an equality with a high probability.  Typically,
the $\star$--square of a random code either fills in its ambient space
or has the same dimension as the symmetric square $S^2 \CC$ of $\CC$
(note that there is a surjective canonical map
$S^2 \CC \rightarrow \CC \star \CC$).  For non symmetric products, one
can observe very similar behaviours. Namely, given two codes
$\CC, \code D \subseteq \Fq^n$, we have
\[
  \dim \CC \star \code D \leq \min \left\{n,\ \dim \CC \cdot \dim
    \code D - {\dim \CC \cap \code D \choose 2}\right\},
\]
where again the inequality is typically an equality when $\CC, \code D$
are random.

On the other hand, the behaviour of AG codes with respect to the $\star$
operation strongly differs from that of random codes.

\begin{thm}\label{thm:prod_AGcodes}
  Let $\X, \cP$ be as in Definition~\ref{def:AGcodes} and
  $F, G$ be two rational divisors of $\X$ then
  \[
    \CL{\X}{\cP}{F} \star \CL{\X}{\cP}{G} \subseteq \CL{\X}{\cP}{F+G}.
  \]
  Moreover, the above inclusion is an equality when $\deg F \geq 2g$
  and $\deg G \geq 2g+1$.
\end{thm}

\begin{proof}
  This is a direct consequence of \cite[Thm. 6]{mumford1970cime}.
  See \cite[Cor. 9]{couvreur2017it}.
\end{proof}

In particular, for a random code $\code R$ whose square does not fill
in the ambient space, according to~(\ref{eq:random_square}) and the
subsequent discussion, the dimension of $\code R \star \code R$ is
quadratic in $\dim \code R$. On the other hand, for an AG code
$\CL{\X}{\cP}{G}$ from a genus $g$ curve $\X$ with
$2g < \deg G < \frac n 2$, Theorem~\ref{thm:prod_AGcodes} together
with Riemann--Roch Theorem entail
\[
  \dim \CL{\X}{\cP}{G} \star \CL{\X}{\cP}{G} = 2 \dim \CL{\X}{\cP}{G} - 1 +g.
\]
In short, for low genera, the dimension of the square of an AG code
is somehow linear in the original dimension instead of being quadratic
for random codes. This significant difference is actually a strength
for AG codes for instance in order to perform an efficient error
correction. On the other hand, as explained further in
\S~\ref{sec:crypto}, it is also a weakness making AG codes bad
candidates for public key encryption. The subsequent sections discuss
several consequences of this observation on the behaviour of AG codes
with respect to the $\star$--product.

 \section{Decoding}\label{sec:decoding}
The design of efficient decoding algorithms for AG codes is an
important part of the theory starting in the late 80's with the work
of Justesen, Larsen, Jensen, Havemose and H{\o}holdt
\cite{justesen89it} and and culminating around 2000 with the works of
Guruswami and Sudan \cite{guruswami99it} on the list decoding of
Reed--Solomon and AG codes. We refer the reader to the excellent
surveys \cite{hoeholdt95it,beelen08} for further details on this
history.

Interestingly, the first works on the decoding of AG codes
\cite{justesen89it,skorobogatov90it} strongly involved the arithmetic
of the underlying curve. A few years later, Pellikaan
\cite{pellikaan92dm} and independently K\"otter \cite{koetter92acct}
noticed that one could get rid from geometry in the algorithm's
description and only take into account the non generic behaviour of AG
codes with respect to the $\star$--product. This point of view leads
to an observation, the very reason why AG codes are equipped with
polynomial time decoders while arbitrary codes are not is their
peculiar behaviour with respect to the $\star$--product.

\subsection{Algebraic decoding in a nutshell}
The decoding paradigm presented by Pellikaan, is now known as {\em Error
  Correcting Pairs}. We give here a rather unusual presentation of this
algorithm that we expect to be somehow more elementary to understand.

Consider a code $\CC$ and a received vector $\yv = \cv + \ev$ where
$\cv \in \CC$ and $\ev \in \Fq^n$ has weight $w_H(\ev) \leq t$,
where $w_H (\cdot)$ is defined in \S~\ref{sec:coding} and $t$ is
a constant that will be clarified further.
Obviously, $\yv$ should be considered as known while $\cv, \ev$ are not.
The key of algebraic decoding is to introduce an auxiliary code
$\code A$ in order to localise the error positions. The code $\code A$
should be chosen so that $\dim \code A > t$ and the product
$\code A \star \CC$ is ``not too large''. In particular
$\code A \star \CC \varsubsetneq \Fq^n$.

The first step of the algorithm consists in
computing the space
\[
  K \eqdef \left\{\av \in \code A ~|~ \av \star \yv \in \code A \star
    \CC \right\}.
\]
This space can be computed by Gaussian elimination with the very knowledge
of $\yv$ and $\code A, \code C$ while, by definition of $\code A \star \CC$
it actually equals
\[
  K = \left\{\av \in \code A ~|~ \av \star \ev\in \code A \star \code
    C \right\}.
\]
Next, the hope is that for any $\av \in K$, the product
$\av \star \ev$ is $0$. This is in particular what should happen
when the minimum distance of $\code A \star \code C$ exceeds $t$ since
$w_H(\av \star \ev) \leq w_H (\ev) \leq t$.  In this situation, any
$\av \in K$ actually vanishes at any nonzero position of $\ev$ and
hence one can expect that the common zero positions of the elements of
$K$ permit to identify the nonzero positions of $\ev$.  Once the
errors are localised, decoding turns out to be solvable by the
resolution of a linear system.

In summary, correcting $t$ errors for a code $\CC$ is possible as soon
as there exists an auxiliary code $\code A$ such that
\begin{enumerate}[(i)]
\item\label{it:dimA} $\dim \code A > t$;
\item\label{it:dminA*B} $d_{\text{min}} (\code A \star \code C) > t$;
\item\label{it:dminA+dminC}
  $d_{\text{min}}(\code A) + d_{\text{min}} (\CC) > n$.
\end{enumerate}

\begin{proof}[Sketch of proof]
  First, note that $K$ contains the so--called {\em shortening}
  of $\code A$ at the error positions, {\em i.e.} the subcode of
  $\code A$ of words vanishing where $\ev$ does not.  Under
  Condition~\eqref{it:dimA}, this shortening of $\mathcal A$ is nonzero
  and hence $K \neq 0$.  Next, because of this previous discussion,
  Condition (\ref{it:dminA*B}) asserts that $K$ actually equals this
  {\em shortening} of $\code A$.  Finally, Condition
  (\ref{it:dminA+dminC}) permits to make sure that, once the error
  positions are known, the resolution of a linear system will provide
  $\ev$ as the unique solution. See \cite[Prop. 2.4]{hoeholdt95it} for
  further details.
\end{proof}

\begin{remark}\label{rem:relax}
  Actually, Condition (\ref{it:dminA*B}) is rather restrictive and
  can be relaxed to
  $\dim \code A - t + \dim \code A \star \code \CC \leq n$ which
  permits to typically expect that
  $(K \star \ev) \cap \code A \star \code C = \{0\}$, which is
  sufficient to obtain that $K$ is the shortening of $\code A$ we look
  for. Similarly, Condition (\ref{it:dminA+dminC}) may be
  removed. Relaxing these conditions is possible at the cost that the
  algorithm may fail for some rare error patterns $\ev$.
\end{remark}

\begin{remark}
  Pellikaan's original description of the so--called {\em Error
    Correcting pairs} involves two auxiliary codes $\code A$ and $\code B$
  (see for instance \cite[Def. 12.1]{hoeholdt95it}). One can actually
  prove that the code $\code B = (\code A \star \code C)^\perp$ satisfies
  the conditions of Error Correcting Pairs.
\end{remark}

\subsection{Back to AG codes}
Now, if the code $\CC$ is the AG code $\CL{\X}{\cP}{G}$, then, given a
divisor $F$ of degree $\geq t + g$, the code $\code A \eqdef \CL{\X}{\cP}{F}$
satisfies all the expected requirements as soon as $n- \deg (F+G) > t$
which leads to
\begin{equation}\label{eq:radius}
  t \leq \frac{d^*-1}{2} - \frac g 2,
\end{equation}
where $d^* \eqdef n - \deg G$ is the {\em Goppa designed distance},
{\em i.e.}  the lower bound on the minimum distance given by
Proposition~\ref{prop:param_AG_codes}.

\subsection{The story continues}
Somehow the ``$-g/2$'' term in \eqref{eq:radius} seems disappointing
since one can expect a unique decoding up to $\frac{d-1}{2}$ errors.
If the previous algorithm is proved to reach the decoding radius
\eqref{eq:radius}, it turns out to be able to correct ``almost any''
error pattern of weight up to half the designed distance (more or less
because of Remark~\ref{rem:relax}). During the 90's, many publications
targeted to get rid of this $g/2$ term. This has been completed in
\cite{ehrhard1993it} and \cite{feng1993it} using two distinct
approaches.  Interestingly, if the radius \eqref{eq:radius} is reached
using a rather generic algorithm which does almost not use the
arithmetic of the underlying curves, the two aforementioned papers
improve the decoding radius while considering again the curve's
arithmetic.

Later on, a new question has been raised: {\em ``can one correct
  errors beyond half the minimum distance at the cost of possibly
returning a list of codewords instead of a single one''}. A positive
answer has been given by Sudan for Reed--Solomon codes
\cite{sudan97jcomp} and Guruswami and Sudan for AG codes
\cite{guruswami99it}.  The latter paper providing a polynomial time
algorithm correcting up to the so--called {\em Johnson bound}, a
radius under which any decoder is guaranteed to return a list whose
size is polynomial in the code length.

The reason why all the aforementioned algorithms succeed to decode AG
codes and not arbitrary ones, can somehow be understood in terms of
the peculiar behaviour of these codes with respect to the $\star$
product. See for instance
\cite{rosenkilde2018,couvreur2020dcc,puchinger2021}.

 \section{Code--based cryptography and McEliece encryption
  scheme}\label{sec:crypto}
Code--based cryptography dates back to 1978 with the seminal paper of
McEliece \cite{mceliece1978dsn} proposing a public key encryption
scheme whose security is related to the hardness of the decoding problem.
Informally, the principle is to take a code for which a decoding
algorithm correcting up to $t$ errors is known and to publish one of
its bases.  Encryption consists in encoding the plaintext into a
codeword and add some randomly chosen error pattern of weight $t$
while decryption consists in decoding, which is a possible task for
the legitimate receiver.

In terms of security, one expects the published version of the
code to be computationally indistinguishable from a random code. Thus,
under this indistinguishability assumption an attacker cannot do
better than trying to solve the decoding problem for an apparently
random code.  The latter problem is proved to be NP--complete
\cite{berlekamp1978it} but more importantly, it is supposed to be hard
{\em in average}, {\em i.e.} hard for randomly chosen inputs. Still,
up to now the best known algorithms for generic decoding have a time
complexity which is exponential in the weight of the error.

Among the various proposals to instantiate McEliece scheme,
Reed--Solomon and AG codes have been proposed
\cite{niederreiter1986pcit,janwa1996dcc}. It turns out that the
algebraic structure of these codes and in particular their structure
with respect to the component wise product mentioned in
\S~\ref{sec:starprod}, which is the key of the algebraic decoding as
explained in \S~\ref{sec:decoding}, also represents a strong weakness
of these codes.  Indeed, because of
Theorem~\ref{thm:prod_AGcodes} and the subsequent discussion, the
operation $\CC \mapsto \CC \star \CC$ precisely provides a polynomial
time distinguisher between AG codes and random ones making the latter
ones unusable for public key encryption. To turn this distinguisher
into a full key recovery attack, we refer the reader to
\cite{couvreur2017it}.

Therefore, if the raw use of AG codes is not a secure, it is still
possible to use ``deteriorate'' versions of them for which decoding is
still possible but with a much lower decoding radius $t$. For
instance, one can consider some random subcodes of AG codes with a
large enough codimension or take {\em subfield subcodes}, that is to
say, start from an AG code in $\F_{q^m}^n$ and only keep vectors whose
entries are all in the subfield $\Fq$. In some sense, McEliece's
original proposal was based on the latter idea with a genus 0
underlying curve.
We refer the reader to \cite[\S~15.7]{couvreur2021}
for further details on the history of cryptanalysis of AG codes.

 \section{Secret sharing}\label{sec:mpc}
In \cite{shamir1979}, A.~Shamir proposes a secure scheme for a set of
players to share a secret so that a coalition of $k$ of them can get
the secret while any coalition of less than $k$ players cannot get
even a partial information on the secret. The scheme is based on
Lagrange interpolation and can actually be interpreted in terms of
codes. For $n-1$ players and a given secret $s \in \Fq$, take an
$[n,k,d]_q$ code $\CC \subseteq \Fq^n$, choose an arbitrary vector
$\cv \in \CC$ whose last entry $c_n = s$ and distribute the $n-1$
other entries to the players. The reconstruction and security
properties of the scheme can actually be interpreted in terms of
the minimum distance of $\CC$ and its dual $\CC^\perp$.

\begin{lem}
  A coalition of $r > n - d_{\text{min}}(\CC)$ players can
  recover the secret while any coalition of
  $r' < d_{\text{min}}(\CC^\perp) - 1$
  players cannot get any information on the secret.
\end{lem}

\begin{proof}
  Given $I \subseteq \{1, \dots, n\}$, denote by $\CC_{I}$ the image
  of $\CC$ under the canonical projection on the entries with
  index in $I$:
  \[\cv \longmapsto \cv_I \eqdef {(c_i)}_{i \in I}~.\] The first assertion is due to the fact that if
  $|I|>n - d_{\text{min}}(\CC)$, then this canonical projection
  $\CC \rightarrow \CC_I$ is injective. Hence, the knowledge of
  $\cv_I$ is enough to recover the whole vector $\cv$ and in
  particular the secret $s = c_n$.

  On the other hand, if $|I| < d_{\text{min}}(\CC^\perp) - 1$, then
  one can prove that $\CC_{I \cup \{n\}} = \Fq^{|I|+1}$ and hence for
  any $s' \in \F_q$ there exists $\cv' \in \CC$ such that
  $\cv'_n = s'$ and $\cv'_I = \cv_I$. Thus, with the only knowledge of
  $\cv_I$, the possible values of the secret are uniformly
  distributed.
\end{proof}

The coding theoretic interpretation of Shamir's scheme rests on a
$k$--dimensional Reed--Solomon code.  It permits the secret recovery by
any $k$--tuple of users while coalitions of $r<k$ users cannot get
any information on the secret. 

Another interesting feature of Shamir's scheme is that it is {\em
  linear}: if one wants to perform linear combinations of secrets, the
users only have to perform separately the linear combination of their
own shares and the result of a legal coalition will be the expected
linear combination. Next, one could expect to be able to perform
further arithmetic operations on the secrets such as
multiplications. In such situations the code $\CC \star \CC$ enters
the game again. This is at the core of the so--called {\em arithmetic
  secret sharing schemes} which can serve as a basis for secure {\em
  multiparty computation} protocols \cite{CDM2000}.  Here we refer the
reader to \cite[Chap.~12]{CDN2015book} and
\cite[\S~15.8.4]{couvreur2021} for further details on the expected
properties a code should have to provide a good arithmetic secret
sharing scheme.  Note that very similarly to other applications,
Reed--Solomon codes --- yielding Shamir's scheme --- are somehow
optimal but suffer from the length restriction: the number of players
cannot exceed the size of the ground field {\em i.e.} the cardinality
of the set of secrets.  Next, AG codes appear to be the best solution
to address this limitation.

In this context, an open question remains: {\em which codes may be
  suitable for arithmetic secret sharing?} At least a first question
could be {\em which codes $\CC$ satisfy the property that
  $\CC \star \CC$ is ``small''?}.  According to
\cite[Thm.~12.19]{CDM2000}, the coding theoretic requirements to
design a multiplicative secret sharing scheme consist in lower bounds
on the minimum distances of $\CC^\perp$ and $\CC \star \CC$. According
to Singleton bound~(\ref{eq:Singleton}), this entails a lower bound on
the dimension of $\CC$ and an upper bound on that of $\CC \star
\CC$. Then, finding codes $\CC$ of fixed dimension $k$ such that
$\CC \star \CC$ has the least possible dimension is a question of
interest and classifying such codes is a crucial question which
remains widely open up to now and that is discussed in the sequel.

\subsection{Which codes have small squares?}
Let us focus on the following problem, {\em which codes
  $\CC \subseteq \Fq^n$ of dimension $k$ have a square of small
  dimension?}  We first need to clarify what is meant by
``small''. For this we need the following definition.
\begin{definition}\label{def:degenerate}
  A code $\code A \subseteq \Fq^n$ is said to be {\em degenerate} if
  it is the direct sum of two subcodes with disjoint supports.
  Equivalently, there exist subcodes
  $\code{A}_1, \code{A}_2 \subseteq \Fq^n$ such that
  $\code A = \code{A}_1 \oplus \code{A}_2$ and
  $\code A_1 \star \code A_2 = \{0\}$.
\end{definition}

Now, according to \cite[Thm.~2.4]{hlx02}, as soon as $\CC \star \CC$
is {\em non degenerate}, there is an analog of Cauchy--Davenport
formula in additive combinatorics which asserts that
\[
  \dim \CC \star \CC \geq 2\dim \CC - 1.
\]
Next, for any positive integer $\gamma$, what can one say about codes
satisfying
\begin{equation}\label{eq:vosper}
  \dim \CC \star \CC = 2 \dim \CC - 1 + \gamma?
\end{equation}
The quantity $\gamma$ should be ``small'' and inspired by additive
combinatorics, one can first target Freiman--like results
\cite[Thm.~5.11]{tao06book} and study the case where
$\gamma \leq \dim \CC - 3$.

First, one can note that AG codes yield a solution of the raised
problem.  Indeed, Theorem~\ref{thm:prod_AGcodes} asserts that an AG
code $\CL{\X}{\cP}{G}$ from a curve of genus $\gamma$ with
$\deg G > 2\gamma$ precisely satisfies this property.  However, they are
not the only ones. It is also possible for some subcodes of AG codes
from curves of genus $<\gamma$ to satisfy \eqref{eq:vosper}. For
instance, for $\gamma =1$, Property (\ref{eq:vosper}) is clearly
satisfied by AG codes from elliptic curves but also by some codes from
a curve of genus $0$ as explained in the example below.

\begin{example}
  Consider the projective line $\P^1$, denote by $P_{\infty}$ the
  place at infinity and by $\cP$ an ordered sequence of distinct degree
  $1$ finite places. The code $\CL{\P^1}{\cP}{kP_{\infty}}$
  is nothing but a $(k+1)$--dimensional Reed--Solomon code and
  let $\CC$ be its subcode spanned by the evaluations of the
  functions $1, x^2, x^3, \dots, x^k$, that is to say, removing $x$.
  The code $\CC\star \CC$ will be spanned by the evaluations of
  $1, x^2, x^3, \dots, x^{2k}$ and hence will have dimension
  $2k = 2 \dim \CC$ while being constructed from a genus $0$ curve.
\end{example}

\begin{question}
  Let $k$ be a positive integer
  and $0 < \gamma \leq k-3$.  Are there $k$--dimensional codes
  $\CC \subseteq \Fq^n$, such that $\CC \star \CC$ is non degenerate,
  satisfies
  \[
    \dim \CC \star \CC = 2 \dim \CC - 1 + \gamma
  \]
  and $\CC$ is {\bf not} contained in an AG code from a curve of genus
  $g\leq \gamma$?
\end{question}

For $\gamma = 0$, it has been proved by Mirandola and Z\'emor
\cite{MZ2015} that MDS codes ({\em i.e.} codes whose minimum distance
reaches the Singleton bound (\ref{eq:Singleton})) which satisfy
$\dim \CC \star \CC = 2 \dim \CC - 1$ are equivalent to Reed--Solomon
codes. Interestingly, this fact was already known from
an algebraic geometry point of view as Castelnuovo's Lemma (see for instance
\cite[\S~III.2,~p.~120]{arbarello1985book}).

For larger $\gamma$'s, the problem is still open. The following
paragraphs discuss a possible geometric approach to study this
problem. We acknowledge that this discussion raises more questions
than answers.

\subsection{The quadratic hull of a code}
At least a first question could be: {\em do such codes always arise
  from curves?}  In \cite{randriam2021} a notion called {\em quadratic
  hull} of a code, associates an algebraic variety to an arbitrary
code in the following manner.  Given a $k$--dimensional code
$\CC \subseteq \Fq^n$, consider a generator matrix $G$ for $\CC$, {\em
  i.e.} a $k \times n$ matrix whose rows span $\CC$.  Assuming that
the columns of $G$ are pairwise independent, and regarding these
columns as points of $\P^{k-1}$, we associate to $\CC$ the
corresponding $n$--tuple $V(\CC)$ of elements of $\P^{k-1}(\Fq)$.  By
this manner, $\CC$ can be regarded as an evaluation code: it is the
code of evaluations of $H^0(\P^{k-1}, \mathcal O(1))$ at the points of
$V(\CC)$.  Moreover, $\CC \star \CC$ can be regarded in this way as
the code of evaluations of $H^0(\P^{k-1}, \mathcal O(2))$ at these
same points! Let $I_2(V(\CC))$ be the space of quadratic forms
vanishing on $V(\CC)$, we have the following exact sequence
\begin{equation}\label{eq:exact}
  0 \longrightarrow I_2(V(\CC)) \longrightarrow
  H^0(\P^{k-1}, \mathcal O(2))
  \longrightarrow \CC \star \CC \longrightarrow 0.
\end{equation}
The {\em quadratic hull} $\HC$ of $\CC$ is the variety defined by the
homogeneous ideal of $\Fq[X_0, \dots, X_{k-1}]$ spanned by
$I_2(V(\CC))$. Equivalently, it is the intersection of quadrics containing
$V(\CC)$.  Clearly, $\HC$ contains $V(\CC)$ but may contain further
elements.

\begin{prop}
  Let $\X$ be a curve of genus $g$, $\cP$ be an ordered $n$--tuple of
  rational points of $\X$ and $G$ be a divisor such that
  $2g+2 \leq \deg G < \frac n 2$. Then, the quadratic hull of
  $\CL{\X}{\cP}{G}$ equals $\X$.
\end{prop}

\begin{proof}
  The proof can be found in \cite[Thm.~2]{MMP2014}, we adapt it here
  to the current notation and terminology.  The condition
  $\deg G \geq 2g+2$ entails the very ampleness of $G$ and, up to
  projective equivalence the corresponding embedding
  $\varphi_G : \X \rightarrow \P(L(G)) \simeq \P^{k-1}$ sends
  precisely $\cP$ onto $V(\CC)$. The curve $\varphi_G (\X)$
  has embedding degree $\deg G$ in $\P^{k-1}$ and hence, any quadric
  hypersurface containing $V(\CC)$ contains the whole $\X$.
  Indeed, if not, the quadric hypersurface would intersect
  $\varphi_G (\X)$ at $n > 2 \deg G$ points (by assumption on
  $\deg G$) which would contradict B\'ezout's Theorem. Therefore, the
  quadratic hull of $\CC$ contains $\varphi_G (\X)$. Next, using Saint
  Donat's result \cite{saintdonat:1972}, such an embedding of $\X$ is
  an intersection of quadrics, which concludes the proof.
\end{proof}

Now the question is, {\em what if $\CC$ is not supposed to arise from a
curve but $\CC \star \CC$ is non degenerate and $\dim \CC \star \CC \leq 2 \dim \CC - 1 + \gamma$ with
$\gamma \leq \dim \CC - 3$?} Is it possible that this code is not
contained in an AG code from a curve of genus $g \leq \gamma$?  The
point is that if $\dim \CC \star \CC$ is small, then 
$\dim I_2(V(\CC))$ is large and one could reasonably expect the quadratic
hull to have a small dimension. Actually, one can prove the following
statement.

\begin{thm}\label{thm:freiman}
  Let $\CC \subseteq \Fq^n$ such that $\CC \star \CC$ is non
  degenerate and $\dim \CC \star \CC \leq 3 \dim \CC - 4$. Assume
  moreover that the quadratic hull $\HC$ of $\CC$ is a reduced absolutely
  irreducible variety. Then, $\HC$ is a curve.
\end{thm}

\begin{remark}
  The above statement is mentioned in \cite[\S~2.11]{randriam2021}
  where the proof is sketched.
\end{remark}

To prove of Theorem~\ref{thm:freiman} we need the following Lemma.

\begin{lem}\label{lem:technical_lemma}
  Under the hypotheses of Theorem~\ref{thm:freiman},
  let $I_2(\HC)$ be the subspace of
  $H^0(\mathbf{P}^{k-1}, \mathcal{O}(2))$ of quadratic forms vanishing on
  $\HC$. Then,
  \[
    I_2(\HC) = I_2(V(\CC)).
  \]
\end{lem}

\begin{proof}
  Using the hypothesis that $\HC$ is reduced, the result is a direct
  consequence of Hilbert Nullstellensatz.
\end{proof}

\begin{proof}[Proof of Theorem~\ref{thm:freiman}]
  Assuming that the quadratic hull is reduced and irreducible, one can
  define its field of functions $K$.  After fixing an arbitrary choice
  of coordinates in $\P^{k-1}$, one can regard
  $H^0(\P^{k-1}, \mathcal O(1))$ as a space of rational fractions and
  there is a canonical restriction map
  \[
    H^0(\P^{k-1}, \mathcal O(1)) \longrightarrow K
  \]
  whose image is denoted by $C$. Since $\CC$ has dimension $k$, the
  $k\times n$ generator matrix $G$ has full rank and hence its columns
  span $\Fq^k$.  Consequently, $V(\CC)$ spans $\P^{k-1}$ or equivalently
  is not contained in any hyperplane. Thus, so does $\HC$
  and hence, the above restriction map is an isomorphim: $C$ has
  dimension $k$ and is isomorphic to $\CC$. Moreover, since $C$
  contains the restrictions to $\HC$ of a coordinate
  system of $\P^{k-1}$, it generates the field $K$ as an
  $\F_q$--algebra.

  Consider the space
  \[
    C \cdot C \eqdef \text{Span} \left\{ uv \in K ~|~ u,v\in C \right\}.
  \]
  We have a canonical surjective map from the symmetric square $S^2 C$
  into $C \cdot C$. Since $C$ is isomorphic to
  $H^0(\P^{k-1}, \mathcal O (1))$, its symmetric square is isomorphic
  to $H^0 (\P^{k-1}, \mathcal O(2))$.  In short, $C \cdot C$ is the
  space of quadratic forms on $\P^{k-1}$ restricted to $\HC$. Hence,
  the kernel of the canonical surjective map
  $S^2 C \rightarrow C \cdot C$ is nothing but the space $I_2(\HC)$ of
  quadratic forms vanishing on $\HC$. By
  Lemma~\ref{lem:technical_lemma}, we deduce that the latter kernel
  $I_2(\HC)$ is nothing but $I_2(V(\CC))$.  In summary, we have the
  following diagram
  \begin{center}
    \begin{tikzpicture}
      \node at (1,1.5) {\(0\)};
      \node at (3,1.5) {\(I_2(V(\CC))\)};
      \node at (6,1.5) {\(S^2 C \)};
      \node at (9,1.5) {\( C \cdot C \)};
      \node at (11,1.5) {\(0 \)};
      \node at (1,0) {\(0\)};
      \node at (3,0) {\(I_2(V(\CC))\)};
      \node at (6,0) {\(H^0(\mathbf{P}^{k-1}, \mathcal O (2))\)};
      \node at (9,0) {\( \CC \star \CC \)};
      \node at (11,0) {\(0\)};
      \draw[->] (1.25,1.52) to (2.2,1.52) ;
      \draw[->] (3.85,1.52) to (5.55,1.52) ;
      \draw[->] (6.5,1.52) to (8.45,1.52) ;
      \draw[->] (9.6,1.52) to (10.8,1.52) ;
      \draw[->] (1.25,0.02) to (2.2,0.02) ;
      \draw[->] (3.85,0.02) to (4.7,0.02) ;
      \draw[->] (7.35,0.02) to (8.45,0.02) ;
      \draw[->] (9.6,0.02) to (10.8,0.02) ;
      \draw[-]  (2.97,0.35) to (2.97,1.2);
      \draw[-]  (3.03,0.35) to (3.03,1.2);
      \draw[->]  (6,1.2) to (6,0.35);
      \node[rotate=90] at (6.15,0.8) {\(\sim\)};
    \end{tikzpicture}
  \end{center}
  Consequently, $C \cdot C$ is isomorphic to $\CC \star \CC$. We
  already mentioned that $C$ generates $K$ as an $\F_q$--algebra and
  just proved that $\dim C\cdot C \leq 3 \dim C - 4$. Then, from
  \cite[Thm.~4.2]{BCZ18}, the transcendence degree of $K$ over $\F_q$
  cannot lie in the range $\{2, \dots, k-1\}$ and hence should be
  $1$. This proves that $\HC$ is a curve.
\end{proof}

As a conclusion, we could get partial answers on the original problem if we
were able to get criteria on the code for its quadratic hull to be a reduced
irreducible variety.

\begin{question}
  Given a code $\CC \subseteq \Fq^n$ such that $\CC \star \CC$ is non
  degenerate and $\dim \CC \star \CC \leq 3 \dim \CC - 4$, under which
  conditions the quadratic hull of $\CC$ is reduced and irreducible?
\end{question}

Note that some conditions should hold for the hull to be irreducible
as suggests the example below.

\begin{example}
  This example appears in \cite[\S~7.9]{randriam2021}.
  Let $\Y \subseteq \P^3$ be the union of a plane and a line not
  contained in it. Let $\cP$ be the set of rational points of $\Y$ and
  consider the code $\CC$ obtained as the set of evaluations of linear
  forms, {\em i.e.} the elements of $H^0(\P^3, \mathcal O (1))$ at the
  elements of $\cP$.  For $q$ large enough, it is simple to observe
  that any quadric containing $\cP$ contains $\Y$. Hence, the quadratic
  hull of the code equals $\Y$. Moreover, the dimension of the space of
  quadratic forms vanishing on $\Y$ is $2$ which entails that
  $\dim \CC \star \CC = \dim H^{0}(\P^3, \mathcal O (2)) - 2 = 8 = 2
  \dim \CC$. On the other hand, for $q$ large enough, one can
  prove that $\cP$ cannot be contained in a curve of genus $\leq 1$
  since $|\cP| = q^2+2q+1$ while Weil's bound asserts that curves
  of genus $\leq 1$ have at most $q + 1 + 2\sqrt q$ rational points.
\end{example}

 \section{A concluding remark}
The present note dealt with codes from curves while one could also
consider codes from higher dimensional varieties.  The latter
topic, without being as investigated in the literature as codes from
curves remains of deep interest: we refer the reader to the survey
\cite{little2008chapter}.
However, our focus was on the behaviour of codes with respect to the
component wise product and codes from curves are very particular from
this point of view. It can be related to the fact that Hilbert
polynomials of curves are linear while they are quadratic for
surfaces. Using this point of view, it seems that one cannot expect to
decode codes from higher dimensional varieties as efficiently as codes
from curves. Similarly, higher dimensional varieties are probably much
less interesting for secret sharing applications.
On the other hand, codes from higher dimensional varieties may remain of deep
interest for other applications such as cloud storage and the design
of locally recoverable codes.

\section*{Acknowledgement}
The author expresses his deep gratitude to Hugues Randriambololona for his
very valuable comments.

\bibliographystyle{alpha}

\begin{thebibliography}{MCMMP14}

\bibitem[ACGH85]{arbarello1985book}
Enrico Arbarello, Maurizio Cornalba, Phillip Griffiths, and Joseph~Daniel
  Harris.
\newblock {\em Geometry of algebraic curves I}, volume 267.
\newblock Springer-Verlag, {F}irst edition, 1985.

\bibitem[ALM{\etalchar{+}}98]{ALMSS98}
Sanjeev Arora, Carsten Lund, Rajeev Motwani, Madhu Sudan, and Mario Szegedy.
\newblock Proof verification and the hardness of approximation problems.
\newblock {\em J. ACM}, 45(3):501--555, May 1998.

\bibitem[BCZ18]{BCZ18}
Christine Bachoc, Alain {Couvreur}, and Gilles {Z{\'e}mor}.
\newblock Towards a function field version of {F}reiman's theorem.
\newblock {\em Algebraic Combin.}, 1(4):501--521, 2018.

\bibitem[B22]{beelen22}
Peter Beelen.
\newblock A survey on recursive towers and Ihara's constant.
\newblock \href{https://arxiv.org/abs/2203.03310}{ArXiv:2203.03310}, 2022.


\bibitem[BH08]{beelen08}
Peter Beelen and Tom H{\o}holdt.
\newblock The decoding of algebraic geometry codes.
\newblock In {\em Advances in algebraic geometry codes}, volume~5 of {\em Ser.
  Coding Theory Cryptol.}, pages 49--98. World Sci. Publ., Hackensack, NJ,
  2008.

\bibitem[Ber68]{B68}
Elwyn~R. Berlekamp.
\newblock {\em Algebraic coding theory}.
\newblock McGraw-Hill Book Co., New York-Toronto, Ont.-London, 1968.

\bibitem[Ber15]{B15}
Elwyn~R. Berlekamp.
\newblock {\em Algebraic coding theory}.
\newblock World Scientific Publishing Co. Pte. Ltd., Hackensack, NJ, revised
  edition, 2015.

\bibitem[BMvT78]{berlekamp1978it}
Elwyn Berlekamp, Robert McEliece, and Henk van Tilborg.
\newblock On the inherent intractability of certain coding problems.
\newblock {\em IEEE Trans. Inform. Theory}, 24(3):384--386, May 1978.

\bibitem[BTV17]{barg2017it}
Alexander {Barg}, Itzhak {Tamo}, and Sergei~G. Vl\u{a}adu\c{t}{}.
\newblock Locally recoverable codes on algebraic curves.
\newblock {\em IEEE Trans. Inform. Theory}, 63(8):4928--4939, 2017.

\bibitem[CCMZ15]{CCMZ2015}
Ignacio Cascudo, Ronald Cramer, Diego Mirandola, and Gilles Z{\'e}mor.
\newblock Squares of random linear codes.
\newblock {\em IEEE Trans. Inform. Theory}, 61(3):1159--1173, 2015.

\bibitem[CDM00]{CDM2000}
Ronald Cramer, Ivan Damg\r{a}rd, and Ueli Maurer.
\newblock General secure multi-party computation from any linear secret-sharing
  scheme.
\newblock In Bart Preneel, editor, {\em Advances in Cryptology --- EUROCRYPT
  2000}, volume 1807 of {\em Lecture Notes in Computer Science}, pages
  316--334. Springer-Verlag Berlin Heidelberg, 2000.

\bibitem[CDN15]{CDN2015book}
Ronald Cramer, Ivan Damg\r{a}rd, and Jesper Nielsen.
\newblock {\em Secure multiparty computation and secret sharing}.
\newblock Cambridge University Press, Cambridge, 2015.

\bibitem[CMCP17]{couvreur2017it}
Alain Couvreur, Irene M{\'a}rquez-Corbella, and Ruud Pellikaan.
\newblock Cryptanalysis of {M}c{E}liece cryptosystem based on algebraic
  geometry codes and their subcodes.
\newblock {\em IEEE Trans. Inform. Theory}, 63(8):5404--5418, Aug 2017.

\bibitem[CP20]{couvreur2020dcc}
Alain Couvreur and Isabella Panaccione.
\newblock Power error locating pairs.
\newblock {\em Des. Codes Cryptogr.}, 88:1561--1593, 2020.

\bibitem[CR21]{couvreur2021}
Alain Couvreur and Hugues Randriambololona.
\newblock Algebraic geometry codes and some applications.
\newblock In W.~Cary Huffman, Jon-Lark Kim, and Patrick Sol\'e, editors, {\em
  Concise Encyclopedia of Coding Theory}. Chapman and Hall/CRC, 2021.

\bibitem[{Ehr}93]{ehrhard1993it}
Dirk {Ehrhard}.
\newblock Achieving the designed error capacity in decoding
  algebraic--geometric codes.
\newblock {\em IEEE Trans. Inform. Theory}, 39(3):743--751, 1993.

\bibitem[FR93]{feng1993it}
Gui-Liang {Feng} and T.~R.~N. {Rao}.
\newblock Decoding algebraic-geometric codes up to the designed minimum
  distance.
\newblock {\em IEEE Trans. Inform. Theory}, 39(1):37--45, Jan 1993.

\bibitem[Gop81]{goppa1981dansssr}
Valerii~D. Goppa.
\newblock Codes on algebraic curves.
\newblock {\em Dokl. Akad. Nauk SSSR}, 259(6):1289--1290, 1981.
\newblock In Russian.

\bibitem[GS95]{GS1995}
Arnaldo Garcia and Henning Stichtenoth.
\newblock A tower of {A}rtin-{S}chreier extensions of function fields attaining
  the {D}rinfeld-{V}ladut bound.
\newblock {\em Invent. Math.}, 121:211--222, 1995.

\bibitem[GS99]{guruswami99it}
Venkatesan Guruswami and Madhu Sudan.
\newblock Improved decoding of {R}eed--{S}olomon and {A}lgebraic--{G}eometry
  codes.
\newblock {\em IEEE Trans. Inform. Theory}, 45(6):1757--1767, 1999.

\bibitem[Ham50]{hamming1950}
Richard~W. Hamming.
\newblock Error detecting and error correcting codes.
\newblock {\em Bell System Technical Journal}, 29:147--160, 1950.

\bibitem[HLX02]{hlx02}
X.~Hou, K.~H. Leung, and Q.~Xiang.
\newblock A generalization of an addition theorem of {K}neser.
\newblock {\em J. Number Theory}, 97:1--9, 2002.

\bibitem[HP95]{hoeholdt95it}
Tom {H{\o}holdt} and Ruud {Pellikaan}.
\newblock On the decoding of algebraic--geometric codes.
\newblock {\em IEEE Trans. Inform. Theory}, 41(6):1589--1614, Nov 1995.

\bibitem[Iha81]{ihara1981jfsuTokyo}
Yasutaka Ihara.
\newblock Some remarks on the number of rational points of algebraic curves
  over finite fields.
\newblock {\em J. Fac. Sci. Univ. Tokyo Sect. IA Math.}, 28:721--724, 1981.

\bibitem[JLJ{\etalchar{+}}89]{justesen89it}
J{\o}rn {Justesen}, Knud~J. {Larsen}, Helge~E. {Jensen}, Allan {Havemose}, and
  Tom {H{\o}holdt}.
\newblock Construction and decoding of a class of algebraic geometry codes.
\newblock {\em IEEE Trans. Inform. Theory}, 35(4):811--821, July 1989.

\bibitem[JM96]{janwa1996dcc}
Heeralal Janwa and Oscar Moreno.
\newblock {McEliece} public key cryptosystems using algebraic-geometric codes.
\newblock {\em Des. Codes Cryptogr.}, 8(3):293--307, 1996.

\bibitem[K{\"o}t92]{koetter92acct}
Ralf K{\"o}tter.
\newblock A unified description of an error locating procedure for linear
  codes.
\newblock In {\em Proceedings Algebraic and Combinatorial Coding Theory III},
  pages 113--117. Hermes, 1992.

\bibitem[Lit08]{little2008chapter}
John~B. Little.
\newblock Algebraic geometry codes from higher dimensional varieties.
\newblock In {\em Advances in algebraic geometry codes}, volume~5 of {\em Ser.
  Coding Theory Cryptol.}, pages 257--293. World Sci. Publ., Hackensack, NJ,
  2008.

\bibitem[McE78]{mceliece1978dsn}
Robert~J. McEliece.
\newblock {\em A Public-Key System Based on Algebraic Coding Theory}, pages
  114--116.
\newblock Jet Propulsion Lab, 1978.
\newblock DSN Progress Report 44.

\bibitem[MCMMP14]{MMP2014}
Irene M\'arquez-Corbella, Edgar Mart{\'\i}nez-Moro, and Ruud Pellikaan.
\newblock On the unique representation of very strong algebraic geometry codes.
\newblock {\em Des. Codes Cryptogr.}, 70(1-2):215--230, 2014.

\bibitem[Mum70]{mumford1970cime}
D.~Mumford.
\newblock Varieties defined by quadratic equations.
\newblock In {\em Questions on algebraic varieties, C.I.M.E., III Ciclo,
  Varenna, 1969}, pages 29--100. Edizioni Cremonese, Rome, 1970.

\bibitem[MZ15]{MZ2015}
Diego Mirandola and Gilles Z{\'e}mor.
\newblock Critical pairs for the product {S}ingleton bound.
\newblock {\em IEEE Trans. Inform. Theory}, 61(9):4928--4937, 2015.

\bibitem[Nie86]{niederreiter1986pcit}
Harald Niederreiter.
\newblock Knapsack-type cryptosystems and algebraic coding theory.
\newblock {\em Problems of Control and Information Theory}, 15(2):159--166,
  1986.

\bibitem[Pel92]{pellikaan92dm}
Ruud Pellikaan.
\newblock On decoding by error location and dependent sets of error positions.
\newblock {\em Discrete Math.}, 106--107:369--381, 1992.

\bibitem[PRS21]{puchinger2021}
Sven Puchinger, Johan Rosenkilde, and Grigory Solomatov.
\newblock Improved power decoding of algebraic geometry codes.
\newblock In {\em 2021 IEEE International Symposium on Information Theory
  (ISIT)}, pages 509--514, 2021.

\bibitem[Ran14]{HR-AGCT14}
Hugues Randriambololona.
\newblock On products and powers of linear codes under componentwise
  multiplication.
\newblock In St\'ephane Ballet, Marc Perret, and Alexey Zaytsev, editors, {\em
  Algorithmic Arithmetic, Geometry, and Coding Theory}, volume 637 of {\em
  Contemp. Math.}, pages 3--77. American Mathematical Society, 2014.

\bibitem[Ran21]{randriam2021}
Hugues Randriambololona.
\newblock The quadratic hull of a code and the geometric view on multiplication
  algorithms.
\newblock In St\'ephane Ballet, Gaetan Bisson, and Irene Bouw, editors, {\em
  Arithmetic, Geometry, Cryptography and Coding Theory}, Contemp. Math., pages
  267--296. American Mathematical Society, 2021.

\bibitem[Ree54]{reed54}
Irving~S. Reed.
\newblock A class of multiple-error-correcting codes and the decoding scheme.
\newblock {\em Transactions of the IRE Professional Group on Information
  Theory}, 4(4):38--49, 1954.

\bibitem[Ros18]{rosenkilde2018}
Johan Rosenkilde.
\newblock Power decoding {R}eed--{S}olomon codes up to the {J}ohnson radius.
\newblock {\em Adv. Math. Commun.}, 12(1):81--106, 2018.

\bibitem[RS60]{reed60siam}
Irving~S. Reed and Gustave Solomon.
\newblock Polynomial codes over certain finite fields.
\newblock {\em Journal of the Society for Industrial and Applied Mathematics},
  8(2):300--304, 1960.

\bibitem[SD72]{saintdonat:1972}
Bernard Saint-Donat.
\newblock Sur les \'equations d\'efinissant une courbe alg\'ebrique.
\newblock {\em C. R. Acad. Sc. Paris}, 274:324--327, 487--489, 1972.

\bibitem[Sha48]{shannon1948}
Claude~E. Shannon.
\newblock A mathematical theory of communication.
\newblock {\em Bell System Technical Journal}, 27:379--423, 623--656, 1948.

\bibitem[Sha79]{shamir1979}
Adi Shamir.
\newblock How to share a secret.
\newblock {\em Commun. ACM}, 22(11):612--613, November 1979.

\bibitem[Sud97]{sudan97jcomp}
Madhu Sudan.
\newblock Decoding of {Reed--Solomon} {C}odes beyond the {E}rror-{C}orrection
  {B}ound.
\newblock {\em J. Complexity}, 13(1):180--193, 1997.

\bibitem[SV90]{skorobogatov90it}
Aleksei~N. {Skorobogatov} and Sergei~G. {Vl\u{a}du\c{t}}.
\newblock On the decoding of algebraic-geometric codes.
\newblock {\em IEEE Trans. Inform. Theory}, 36(5):1051--1060, Sep. 1990.

\bibitem[TB14]{tamo2014it}
Itzhak {Tamo} and Alexander {Barg}.
\newblock A family of optimal locally recoverable codes.
\newblock {\em IEEE Trans. Inform. Theory}, 60(8):4661--4676, 2014.

\bibitem[TV06]{tao06book}
T.~Tao and V.~H. Vu.
\newblock {\em Additive combinatorics}, volume 105 of {\em Cambridge studies in
  advanced mathematics}.
\newblock Cambridge University Press, 2006.

\bibitem[TVZ82]{tsfasman1982mn}
Michael~A. Tsfasman, Serge~G. Vl{\u{a}}du{\c{t}}, and Th. Zink.
\newblock Modular curves, {S}himura curves, and {G}oppa codes, better than
  {V}arshamov-{G}ilbert bound.
\newblock {\em Math. Nachr.}, 109:21--28, 1982.

\bibitem[VD83]{vladut1983faa}
Sergei~G. Vl\u{a}du\c{t}{} and Vladimir~G. Drinfeld.
\newblock Number of points of an algebraic curve.
\newblock {\em Funct. Anal. Appl.}, 17:53--54, 1983.

\bibitem[WB83]{WB83}
Lloyd~R. Welch and Elwyn~R. Berlekamp.
\newblock Error correction for algebraic block codes, 1983.
\newblock US patent number 4,633,470.

\end{thebibliography}
\newcommand{\etalchar}[1]{$^{#1}$}

\end{document}